\documentclass[a4paper,12pt]{article}
\usepackage{amsmath,amssymb,euscript,amsthm,amsfonts}
\newcommand{\rmQ}{\mbox{\rm Q}}
\newcommand{\ovrmQ}{\ov{\mbox{\rm Q}}}

\newcommand{\rmd}{\mbox{\rm d}}
\newcommand{\mbR}{{\mathbb R}}
\newcommand{\mbN}{{\mathbb N}}

\newcommand{\cX}{{\mathcal X}}
\newcommand{\BL}{\mathop{\rm BL}}
\newcommand{\ov}{\overline}

\newcommand{\vf}{\varphi}
\newcommand{\ve}{\varepsilon}
\theoremstyle{plain}
\newtheorem{thm}{Theorem}
\newtheorem{lem}{Lemma}
\newtheorem{cor}{Corollary}

\begin{document}\large
\renewcommand{\proofname}{Proof}
\begin{center}
Andriy YURACHKIVSKY\\[0,5cm]
{\bf A Criterion for Precompactness in the Space of Hypermeasures
\footnote{This is the translation, with minor amendments, of the
communication published in Ukrainian in {\it Dopovidi
Natsionalno\u{\i}i Akademii Nauk Ukrainy} [{\it Reports Nat. Acad.
Sci. Ukr.}], 2006, No.9, p.38--41.}}
\end{center}
\vskip0,5cm
\begin{quote}
{\small{\bf Abstract.} Let $\rmQ$ denote the space of signed
measures on the Borel $\sigma$-algebra of a separable complete
space $X$. We endow $\rmQ$ with the norm $\|q\|=\sup|\int\vf\rmd
q|$, where the supremum is taken over all Lipschitz with constant
1 functions whose module does not exceed unity. This normed space
is incomplete provided $X$ is infinite and has at least one limit
point. We call its completion the space of hypermeasures.
Necessary and sufficient conditions for precompactness (=relative
compactness) of a set of hyper\-meas\-ures are found. They are
similar to those of Prokhorov's and Fernique's theorems for
measures.}
\end{quote}

\vskip0,5cm {\small{\bf Keywords:} completion, hypermeasure,
quasicontinuous functional, equiquasicontinuity, precompactness,
tightness.}

\vskip0,5cm

Let $(X, \rho)$ be a metric space. We introduce the notation:
$\cX$ -- the $\sigma$-algebra of Borel sets in $X;$
$$
J(f)=\sup_{x,y\in X,\ x\ne y}\frac{|f(x)-f(y)|}{\rho(x,y)};
$$
$$
\Phi\equiv\Phi(X,\rho)=\{\vf\in\mbR^X:\ \forall\ x,y\in X \
|\vf(x)|\leq1 \ \& \ |\vf(x)-\vf(y)|\leq\rho(x,y)\};
$$
$\BL\equiv\BL(X, \rho)=\mathop{\cup}\limits_{a>0}a\Phi$ -- the
class of all bounded Lipschitz functions on $X; \
\rmQ\equiv\rmQ(X)$ -- the set of all charges (=signed measures) on
$\cX;$ for $q\in\rmQ$ \ $qf=\int f\rmd q$ $(f\in {\rm L}_1(X,\cX,
|q|)$, the integration is performed over $X$),
$\|q\|_\rho=\mathop{\sup}\limits_{\vf\in\Phi(X,\rho)}|q\vf|$.

The class BL contains all the functions
$\vf_{x_0}(x)=\rho(x_0,x)/(1+\rho(x_0,x))$ and therefore separates
the points of $X$. Then Lemma 1 \cite{U} asserts that
$\|\cdot\|_\rho$ is a norm in $\rmQ$ provided the space $(X,
\rho)$ is complete and separable. We denote the completion of
$\rmQ$ w.r.t. this norm by $\ovrmQ$ and call its elements {\it
hypermeasures}. This definition does not imply that the space
$\rmQ$ is incomplete -- but it will be so if the set $X$ is
infinite and has at least one limit point \cite[p. 246]{Bog}. (In
\cite{U} the last condition was missed.) Every hypermeasure can be
realized as a linear functional on $\BL$ \cite{U}. So we will say
``hypermeasure on $\BL(X)$''. The value of a hypermeasure $t$ on a
function $\vf$ will be denoted, in the same manner as for charges,
$t\vf.$ The convergence in $\ovrmQ$ is that w.r.t. the norm.

The goal of this communication is to find the necessary and
sufficient conditions for precompactness (=relative compactness)
of an arbitrary set of hyper\-meas\-ures. For measures, such
conditions are provided by the classical Prokhorov's theorem
\cite{VTCh} widely used in probability theory. A generalization of
this theorem for Lusin spaces was proved by Fernique \cite{F}.
\begin{thm}
\label{thm:1} Let the space $X$ be separable and complete. Then
for any convergent sequence $(t_n)$ of hypermeasures and any
uniformly bounded pointwise converging to zero sequence  $(\vf_n)$
of functions from $\BL$ such that $\\ \sup\limits_n
J(\vf_n)<\infty$ the relation $t_n\vf_n\to0$ holds.
\end{thm}
\begin{proof}
Consider first a stationary sequence of hypermeasures:
$t_n=t\in\ovrmQ$. Let $(\vf_n)$ be such a sequence of functions
from $\BL$ that $$M\equiv
\sup\limits_n\left(\sup\limits_x|\vf_n(x)|
+J(\vf_n)\right)<\infty$$ and $\vf_n\rightarrow0.$ Assume that
$t\vf_n\not\rightarrow0.$ Then there exists a number $a>0$ such
that $|t\vf_k|\geq2a$ for infinitely many $k$. By the definition
of hypermeasure there exists a charge $q$ such that
\begin{equation}
\label{eq:1} \|t-q\|<\ve.
\end{equation}
By the dominated convergence theorem $q\vf_n\to0,$ so
$|q\vf_k-t\vf_k|>a$ for infinitely many $k$. On the other hand,
$|q\vf_k-t\vf_k|\leq M\|q-t\|,$ which together with the previous
inequality yields $\|q-t\|> a/M$. This contradiction with
\eqref{eq:1} proves the theorem in the case $t_n=t.$

Let now the sequence $(t_n)$  converge to some hypermeasure $t.$
By the choice of $(\vf_n)$,  $\vf_n/M\in \Phi$ for all  $n$. So
$|(t_n-t)\vf_n|\leq M\|t_n-t\|\to0.$ Now, the identity
$t_n\vf_n=(t_n-t)\vf_n+t\vf_n$ reduces the general case to that
considered above.
\end{proof}

Theorem \ref{thm:1} asserts the property of hypermeasures somewhat
weaker (because of condition \eqref{eq:1}) than the continuity
property of  charges. Let us formulate it in a more general form.

We say that a set $T$  of linear functionals on $\BL(X)$ is {\it
equiquasicontinuous} if for any $\ve>0$  there exists a Tykhonov's
neighborhood of zero $U\in\mbR^X$ such that
\begin{equation}
\label{eq:2} \sup_{t\in T,\ \vf\in U\cap\Phi}|t\vf|<\ve.
\end{equation}
If herein $T$ is a singleton: $T=\{t\}$, then the functional $t$
will be called {\it quasicontinuous}. In this terminology, Theorem
\ref{thm:1} asserts that every convergent sequence of
hypermeasures is equiquasicontinuous, in particular, every
hypermeasure is a quasicontinuous functional.
\begin{lem}
\label{lem:1'} Let the space $X$ be separable and complete. Then
in order that a set $T$ of hyper\-mea\-sures on $\BL(X)$ be
equiquasicontinuous it is necessary and sufficient that for any
sequence $(t_n)\in T^{\mbN}$ and any pointwise converging to zero
sequence $(\vf_n)\in\Phi^{\mbN}$ the relation $t_n\vf_n\to0$ hold.
\end{lem}
\begin{proof}
{\tt Sufficiency.} Separability of $X$ and the definition of the
class $\Phi$ imply existence of a decreasing sequence $(U_n)$ of
Tykhonov's neighborhoods of zero in $\mbR^X$ such that
\begin{equation}
\label{eq:2'} \mathop{\bigcap}\limits_n U_n\cap\Phi=\{0\}.
\end{equation}
If $T$  is not equiquasicontinuous, then one can find $\ve>0$ and,
for each $n$, a hypermeasure $t_n\in T$  together with a function
$\vf_n\in U_n\cap\Phi$ such that
\begin{equation}
\label{eq:3}
|t_n\vf_n| \geq\ve.
\end{equation}
Then the sequence $(\vf_n)$ pointwise converges to zero, but
$t_n\vf_n\not\rightarrow0.$

{\tt Necessity.} Let a set $T$ be equiquasicontinuous. Let us fix
$\ve>0$ and choose a Tykhonov's neighborhood of zero
$U\subset\mbR^X$ such that the inequality \eqref{eq:2} holds. If a
sequence  $(\vf_n)\in\Phi^{\mbN}$ pointwise converges to zero,
then ultimately all its members belong to $U\cap\Phi$ and
therefore for any sequence $(t_n)\in T^{\mbN}$ \ $|t_n\vf_n|<\ve$
if $n$ is sufficiently large.
\end{proof}
The following main result is an analogue of the above-mentioned
Fernique's theorem \footnote {This analogy was noticed by the
author when the Ukrainian original of the present communication
had been already published. The paper \cite{F} is not cited
there.} and a generalization of Prokhorov's theorem where the
tightness condition is re-formulated in Fernique's form
(equivalent to the classical one if $X$ is Polish).
\begin{thm}
\label{thm:2}Let the space $X$ be separable and complete. Then in
order that a set $T\subset\ovrmQ(X)$ be precompact it is necessary
and sufficient that it be bounded and equiquasicontinuous.
\end{thm}
\begin{proof}
{\tt Necessity.} Let us take an arbitrary sequence $(U_n)$  of
Tykhonov's neighborhoods of zero with property \eqref{eq:2'}. If
the set $T$ is not equiquasicontinuous, then one can find $\ve>0$
such that any neighborhood $U_n$ does not satisfy condition
\eqref{eq:2}. And this means that for each $n$ there exist $t_n\in
T$ and $\vf_n\in U_n\cap\Phi$ such that inequality \eqref{eq:3}
holds.

By the choice of $(U_n)$  and condition \eqref{eq:2'} $\vf_n(x)\to
0$ for any $x$. Since the set $T$  is by assumption precompact,
the sequence $(t_n)$ contains a convergent subsequence $(t_n, n\in
S\subset\mbN).$ Then by theorem \ref{thm:1} $t_n\vf_n\to0$ as
$n\to\infty, n\in S, $ which contradicts to \eqref{eq:3}.

Necessity of boundedness is obvious.

{\tt Sufficiency.} Let us take an arbitrary countable dense subset
$\Psi$ of the set $\Phi.$ Since by assumption $\sup_{t\in
T}\|t\|<\infty,$ any sequence $(t_n)\in T^{\mbN}$ contains a
subsequence $(t_n, n\in S\subset\mbN)$ such that for any
$\psi\in\Psi$ the numeral sequence $(t_n\psi, n\in S)$ converges.
Hence and from equiquasicontinuity of $T,$ writing
$$
|t_m\vf-t_n\vf|\leq|t_m(\vf-\psi)|+|t_m\psi-t_n\psi| +|t_n(\psi-\vf)|,
$$
we deduce fundamentality and therefore convergence of the sequence
$(t_n\vf, n\in S)$ for any $\vf\in\BL.$ Let us denote
$t_0\vf=\mathop{\lim}\limits_{n\to\infty,\ n\in S}t_n\vf$ and show
that for any sequence $(\eta_n, n\in S)$ of functions from $\Phi$
\begin{equation}
\label{eq:4} t_n\eta_n-t_0\eta_n\to0 \ \mbox{as} \ n\to\infty,
n\in S.
\end{equation}

Separability of $X$ and the definition of $\Phi$ imply that for an
arbitrary infinite set $S_1\subset S$ there exist an infinite set
$S_2\subset S_1$ and a function $\eta\in\Phi$ such that for any
$x$ $\eta_n(x)\to\eta(x)$ as $n\to\infty, n\in S_2.$
Equuiquasicontinuity of $T$ and the definition of $t_0$ entail
quasicontinuity of the latter. Hence, writing
$$
t_n\eta_n-t_0\eta_n=t_n(\eta_n-\eta)+t_n\eta-t_0\eta+t_0(\eta-\eta_n)
$$
and recalling once again the definition of $t_0,$  we get by
Theorem \ref{thm:1} relation \eqref{eq:4} with $S_2$  instead of
$S.$ Since the infinite set $S_1\subset S$ which $S_2$ was
extracted from is arbitrary, it holds for $S_2=S$, too. And this
in view of arbitrariness of the sequence $(\eta_n)\in\Phi^S$ means
that the sequence $(t_n, n\in S)$ is fundamental and therefore
(recall that the space $\ovrmQ$ is by construction complete)
convergent.

So, any sequence in $T$  contains a convergent subsequence, which
means that $T$ is precompact.
\end{proof}

Theorem \ref{thm:2} is applicable, in particular, to measures
(here and below -- finite). Let us show that in his case it turns
to the above-mentioned Prokhorov's criterion. We say, somewhat
extending the notion of tightness,  that a set $T$ of measures on
${\cX}$ is {\it tight}, if for any $\ve
>0$ there exists a completely bounded set $K\subset X$  such that
for all $q\in T\quad q(X\setminus K)<\ve$. If $X$ is complete,
then this definition is equivalent to the conventional one
demanding compactness of $K$.

\begin{lem}
\label{lem:2}Let $T$ be a bounded tight set of measures on the
$\sigma$-algebra of Borel sets in a metric space $X$. Then $T$ is
equiquasicontinuous.
\end{lem}

\begin{proof}
It suffices to show that for any $\ve >0$ and completely bounded
set $K\subset X$ there exists a Tykhonov's neighborhood of zero
$U\subset \mbR ^X$ such that for all $q\in T$ ³ $\vf\in U\cap\Phi$
the inequality $\left|\int_K \vf\rmd q\right|<\ve$ holds.

By assumption there exists $C>0$ such that $q(X)\le C$ for all
$q\in T$. So $U$ will possess the required property if
$$\sup\limits_{\vf\in U\cap\Phi}\sup\limits_{x\in K}|\vf(x)|<\frac{\ve}{C}\ .$$
Let us take an existing by the choice of $K$ finite set $A\subset
X$ such that each point $x\in K$ is less than $\ve /2C$ apart from
some point $a(x)\in A$. Then, by the definition of the class
$\Phi$, for any $\vf\in\Phi$ and $x\in K\quad
|\vf(x)-\vf(a(x))|<\ve /2C$, so that one may put $U=\{f\in\mbR
^X:\forall a\in A\ |f(a)|<\ve /2C\}$.
\end{proof}

It is known \cite{H}, that the metric $\Lambda
(q_1,q_2)\equiv||q_1-q_2||$ induces the $\ast$-weak convergence in
the space of measures on $\cX$ (in \cite{H}, the subspace of
probability measures is considered, but the same argument applies
to the whole space). Consequently, precompactness of a set $T$ of
measures on $\cX$ is tantamount to the following property: for
each sequence $(q_n)\in T^{\mbN}$ there exist a measure $q$ íà
$\cX$ and an infinite set $S\subset\mbN$ such that for any
$f\in{\rm C_b}(X)\quad q_nf\to qf$ as $n\to\infty ,\ n\in S$. Then
Prokhorov's theorem in the necessity part asserts that under the
condition of completeness and separability of $X$ any precompact
w.r.t. the norm of the space $\rmQ$ set of measures is tight. This
together with Theorem \ref{thm:2} and  Lemma \ref{lem:2} leads us
to the following conclusion.

\begin{cor}
\label{cor:1} Let the space $X$ be separable and complete. Then in
order that a bounded set of measures on $\cX$ be
equiquasicontinuous it is necessary and sufficient that it be
tight.
\end{cor}

\vskip0,5cm

\end{document}